\documentclass{amsart}   

\usepackage[T1]{fontenc}
\usepackage[utf8]{inputenc}
\usepackage[english]{babel}
\usepackage{amsmath}
\usepackage{amssymb}
\usepackage{amsfonts}
\usepackage{amsxtra}
\usepackage{enumerate}
\usepackage{verbatim}
\usepackage{color}
\usepackage[mathscr]{eucal}
\usepackage[pdftex,colorlinks,urlcolor=blue,pdfstartview=FitH,
]{hyperref}

\usepackage{setspace}

\newcommand{\N}{\mathbb N}
\newcommand{\Z}{\mathbb Z}

\newcommand{\R}{\mathbb R}

\newtheorem{thm}{Theorem}[section]
\newtheorem{definition}[thm]{Definition}
\newtheorem{lemma}[thm]{Lemma}
\newtheorem{remark}[thm]{Remark}

\newtheorem{proposition}[thm]{Proposition}
\newtheorem{question}[thm]{Question}
\newtheorem{conjecture}[thm]{Conjecture}
\newtheorem*{remark*}{Remark}

\newtheorem*{fact}{Fact}

\let\phi=\varphi
\let\e=\varepsilon

\numberwithin{equation}{section}
\usepackage{graphicx}
\begin{document}

\title{A Riemannian plane with only two injective Geodesics} 

\author{Victor Bangert}
\address{Mathematisches Institut, Abteilung f\"ur Reine Mathematik, Albert-Ludwigs-Universit\"at Freiburg, Ernst-Zermelo-Stra\ss e 1, 79104 Freiburg, Germany}
\email{Victor.Bangert@math.uni-freiburg.de}
\thanks{}

\author{Stefan Suhr}
\address{Fakult\"at f\"ur Mathematik, Ruhr-Universit\"at Bochum, Universit\"atsstra\ss e 150, 44780 Bochum, Germany}
\email{Stefan.Suhr@ruhr-uni-bochum.de}
\thanks{Stefan Suhr is partially supported by the SFB/TRR 191 ``Symplectic Structures in Geometry, Algebra and Dynamics'', funded by the Deutsche Forschungsgemeinschaft.}


\begin{abstract}
We present an example of a complete Riemannian plane with precisely two injective geodesics -- up to reparameterization. The example arises as a perturbation of a surface of revolution 
with contracting end. The last section is devoted to open problems.
\end{abstract}

\maketitle

\onehalfspacing

\section{Introduction}\label{sec1}

The geodesics of a complete, noncompact Riemannian manifold $M$ can be divided into the following three classes:

\begin{definition}
A geodesic $c\colon \R\to M$ is either 
\begin{itemize}
\item[(i)] proper (as a map from $\R$ to $M$), or
\item[(ii)] bounded, i.e $c(\R)$ is a bounded subset of $M$, or
\item[(iii)] oscillating, i.e. neither proper nor bounded.
\end{itemize}
\end{definition}
Proper geodesics are alternatively called ``divergent'' or ``escaping'' in the literature. 

Since the end of the $19^\text{th}$ century questions concerning the existence and quantity of geodesics in one of these classes have found continuous interest in differential geometry
and in the theory of dynamical systems, see the contributions by H.v. Mangoldt \cite{Mangoldt81} and by J. Hadamard \cite{Hadamard1898} as early examples for this interest.

Another important property that a geodesic $c\colon \R\to M$ may or may not have is injectivity. This is of particular interest if $\dim M=2$. Note that we require injectivity of $c$ as a 
map defined on $\R$, in particular periodic geodesics will not be injective. A geodesic $c\colon \R\to M$ is proper and injective if and only if $c$ is an embedding. Adopting the 
terminology introduced in \cite{Cardelell19} we define:

\begin{definition}
An embedded geodesic line in $M$ is the image $c(\R)$ of an injective, proper geodesic $c\colon \R\to M$. 
\end{definition}

\noindent Note that embedded geodesic lines need not be length-minimizing. Conversely, if a geodesic $c\colon \R\to M$ is length-minimizing between any two of its points, then $c(\R)$ is an 
embedded geodesic line. Such lines will be called {\it straight lines}. 

Our main result sheds light on the following question: 

\begin{question}
What is the maximal number $n\in \N\cup \{\infty\}$ such that every complete Riemannian plane has at least $n$ embedded geodesic lines?
\end{question}

\noindent It is proved in \cite{BanCom81} that $n\ge 1$, i.e. every complete Riemannian plane contains an embedded geodesic line. Here we show:

\begin{thm}\label{T1}
There exists a complete Riemannian plane with precisely two embedded geodesic lines.
\end{thm}

\noindent So Theorem \ref{T1} implies $n\le 2$. We have the following conjecture, cf. also the end of \cite{Ban85}:

\begin{conjecture}\label{C1.5}
Every complete Riemannian plane has at least two embedded geodesic lines.
\end{conjecture}

\noindent If this conjecture is true we conclude from Theorem \ref{T1} that $n=2$. 

Actually we prove a little more than is stated in Theorem \ref{T1}, namely the existence of a complete Riemannian plane with only two injective geodesics - up to reparameterization. 
This example is constructed by a compactly supported perturbation of the rotationally invariant metric $g$ on $\R^2$ described below. 

\noindent Throughout the paper we will parameterize $\R^2\setminus \{(0,0)\}$ by standard polar coordinates 
\[
(r,\theta)\colon \R^2\setminus \{(0,0)\}\to (0,\infty)\times S^1,
\]
where $S^1=\R/2\pi \Z$.\footnote{We follow the time-honored abuse of notation and denote coordinates and coordinate functions by the same symbols.}
The coordinate function $r$ extends to $(0,0)$ by $r(0,0)=0$. In these coordinates the unperturbed metric $g$ is given by
\[
g=dr^2+f^2\circ r d\theta^2,
\]
where $f\in C^\infty(\R_{>0},\R_{>0})$ satisfies 
\begin{equation}\label{E1.1}
f(r)=\sin r\text{, if } r\in \left(0,\frac{3\pi}{4}\right)\text{,}
\end{equation}
and
\begin{equation}\label{E1.2}
f'(r)<0 \text{, if } r>\frac{\pi}{2}\text{, and }\lim_{r\to \infty} f(r)=0.
\end{equation}
Note that on the set $\left(0,\frac{3\pi}{4}\right)\times S^1$ the metric $g$ coincides with the metric of the round sphere of radius one in geodesic polar coordinates. 
In particular, $g$ extends to a smooth metric on $\R^2$, which is obviously complete. To avoid confusion with angles on $S^1$, we set $R_0:=\frac{\pi}{2}$ and 
$R_1:= \frac{3\pi}{4}$. In Section \ref{sec2} and \ref{sec3} we will prove

\begin{proposition}\label{P1.6}
There exists $\e_0>0$ and a $C^\infty$-family of Riemannian metrics $g_\e$, $\e\in (-\e_0,\e_0)$, such that the following is true:
\begin{itemize}
\item[(i)] $g_0=g$ and $g_\e$ and $g$ coincide outside the set 
\[
\{(r,\theta)|\; R_0<r<R_1,\; 0<\theta<\pi\}.
\]
\item[(ii)] If $0<|\e|<\e_0$ then $(\R^2,g_\e)$ has precisely two injective geodesics - up to reparameterization.
\end{itemize}
\end{proposition}

The proof of Proposition \ref{P1.6} depends primarily on the scattering method from \cite{GlSi78}. The metric $g_\e$ will scatter the $g$-geodesics hitting the circle $\{R_1\}\times S^1$
orthogonally at the point $(R_1,\theta)$ so as to intersect the circle $\{R_0\}\times S^1$ at the point $(R_0,\theta)$ under the angle $\frac{\pi}{2}+\phi_\e(\theta)$ determined by a 
deflection function $\phi_\e\colon S^1\to \left(-\frac{\pi}{4},\frac{\pi}{4}\right)$, see Section \ref{sec2} for details. 

Finally we note that the special properties of the metric $g$, in particular condition \eqref{E1.1}, are requested in order to simplify the proof of Proposition \ref{P1.6}. In principle, one 
could prove this perturbation result for a much larger class of rotationally symmetric metrics. But since the only purpose of Proposition \ref{P1.6} is to prove Theorem \ref{T1} we tried to 
minimize the necessary arguments. 

Throughout the paper geodesics will be parameterized by arclength.

\section{The perturbed metrics $g_\e$}\label{sec2}

The construction of the metrics $g_\e$ follows the method used to prove the more general Theorem 1 in \cite{GlSi78}. We choose a function $\phi \in C^\infty(S^1,\R)$ 
satisfying the following conditions:
\begin{equation}\label{E2.1}
\phi\left(\frac{\pi}{2}+\theta\right)=-\phi\left(\frac{\pi}{2}-\theta\right)\text{, for all }\theta\in \left(0,\frac{\pi}{2}\right)
\end{equation}
\begin{equation}\label{E2.2}
\phi^{-1}(0)=\left\{\frac{\pi}{2}\right\}\cup [\pi,2\pi]
\end{equation}
The functions $\phi_\e =\e\phi$ will play the role of the deflection functions mentioned at the end of Section \ref{sec1}. Note that \eqref{E2.2} implies:
\begin{equation}\label{E2.3}
\text{If }\e\neq 0\text{ and }\theta\in \left(0,\frac{\pi}{2}\right)\cup  \left(\frac{\pi}{2},\pi\right)\text{ then }\phi_\e(\theta)\neq 0.
\end{equation}
As a consequence of \eqref{E2.1} and \eqref{E2.2} we have 
\[
\int_0^{2\pi}\sin\phi_\e(\theta)d\theta=0. 
\]
Hence we can define a function $l_\e\in C^\infty(S^1,\R)$ by 
\begin{equation}\label{E2.4}
l_\e(\theta):=\frac{\pi}{4}-\int_0^\theta \sin \phi_\e(\theta) d\theta. 
\end{equation}
We will always assume that $|\e|$ is so small that $|\phi_\e|<\frac{\pi}{4}$ and $l_\e>0$. Note that 
\begin{equation}\label{E2.5}
l_\e(\theta)=\frac{\pi}{4}\text{ for }\theta\in [\pi,2\pi]\text{, and }l_\e(\theta)\neq \frac{\pi}{4}\text{ if }\e\neq 0\text{ and }\theta\in (0,\pi).
\end{equation}
Recall that we have set $R_0=\frac{\pi}{2}$, $R_1= \frac{3\pi}{4}$. In particular we have $R_1-R_0=\frac{\pi}{4}$.

\begin{remark}
Suppose $g_\e$ is a metric such that the $g_\e$-geodesics intersecting $\{R_1\}\times S^1$ orthogonally, foliate $[R_0,R_1]\times S^1$ and such that the 
$g_\e$-geodesic starting at $(R_1,\theta)$ orthogonally to $\{R_1\}\times S^1$ hits $\{R_0\}\times S^1$ at $(R_0,\theta)$ with angle $\frac{\pi}{2}+\phi_\e(\theta)$. 
Let $L_\e(\theta)$ denote the $g_\e$-length of this geodesic connecting $(R_0,\theta)$ to $(R_1,\theta)$. Then the first variation formula implies that $L_\e'(\theta)=-\sin\phi_\e(\theta)$,
i.e. there exists $c\in \R$ such that $L_\e=l_\e+c$. In particular, we see that the condition 
\[
\int_0^{2\pi} \sin\phi_\e(\theta)d\theta=0
\]
is necessary for the existence of $g_\e$. 
\end{remark}

\begin{proposition}\label{P2.2}
There exists $\e_0>0$ and for every $\e\in (-\e_0,\e_0)$ a Riemannian metric $g_\e$ on $\R^2$ and a diffeomorphism 
\[
\Phi_\e\colon \{(s,\theta)|\; \theta\in S^1, 0\le s \le l_\e(\theta)\}\to [R_0,R_1]\times S^1
\]
satisfying the following properties:
\begin{itemize}
\item[(i)] $g_0=g$ and $\Phi_0(s,\theta)=(R_0+s,\theta)$ for all $(s,\theta)\in \left[0,\frac{\pi}{4}\right]\times S^1$.
\item[(ii)] $g_\e$ and $\Phi_\e$ are smooth as functions of all three variables.
\item[(iii)] $g_\e$ and $g$ coincide outside the set $(R_0+\e_0,R_1-\e_0)\times (0,\pi)$. 
\item[(iv)] For every $\theta\in S^1$ the curve 
\[
c_{\e,\theta}\colon [0,l_\e(\theta)]\to [R_0,R_1]\times S^1,\quad c_{\e,\theta}(s)=\Phi_\e(s,\theta),
\]
is a $g_\e$-geodesic satisfying 
\[
\dot{c}_{\e,\theta}(0)=\cos\phi_\e(\theta)\partial_r|_{(R_0,\theta)}+\sin\phi_\e(\theta) \partial_\theta|_{(R_0,\theta)},
\]
and $\dot{c}_{\e,\theta}(l_\e(\theta))=\partial_r|_{(R_1,\theta)}$, and $(r\circ c_{\e,\theta})^{\boldsymbol{\cdot}}>0$.
\item[(v)] For $\theta=\frac{\pi}{2}$ we have $c_{\e,\frac{\pi}{2}}([0,l_\e(\frac{\pi}{2})]=[R_0,R_1]\times \{\frac{\pi}{2}\}$.
\end{itemize}
\end{proposition}

Note that (iv) implies $\Phi_\e(0,\theta)=(R_0,\theta)$ and $\Phi_\e(l_\e(\theta),\theta)=c_{\e,\theta}(l_\e(\theta))=(R_1,\theta)$. Except for differences in terminology 
the proof of Proposition \ref{P2.2} follows from the arguments given in \cite{GlSi78}. We present a complete proof of Proposition \ref{P2.2} in the appendix since some
properties of the construction that are relevant for our application are not explicitly mentioned in \cite{GlSi78}.

\section{Injective $g_\e$-geodesics}\label{sec3}

In this section we prove Proposition \ref{P1.6} and, consequently, Theorem \ref{T1}. We rely on the family of metrics $g_\e$ from Proposition \ref{P2.2}. First we collect some well known 
facts concerning the geodesics of the unperturbed metric $g=g_0$. Since $g=dr^2+f^2\circ r d\theta^2$ is invariant under rotations, i.e $\partial_\theta$ is a Killing 
vector field for $g$, the geodesic flow of $g$ has a second integral, namely Clairaut's integral $I$. We define $I$ on the unit tangent bundle $S\R^2$ of $g$ by 
\[
I\colon S\R^2\to [-1,1],\quad I(v)=g(\partial_\theta,v).
\]
There are the following types of $g$-geodesics $c\colon \R\to \R^2$: 
\begin{itemize}
\item[(3.1)]
If $I(\dot{c})=0$ then $c$ is orthogonal to all the circles $\{r\}\times S^1$, $r>0$, and, modulo translation of the parameter, we may assume that $r\circ c(0)=0$. Then 
there exists $\theta\in S^1$ such that $c$ is given in polar coordinates by 
\[
c(s)=\begin{cases} (s,\theta),&\text{ for }s>0\\ (-s,\theta+\pi),&\text{ for }s<0.
\end{cases}
\]
These geodesics will be called {\it radial}. 
\item[(3.2)] If $|I(\dot{c})|=1$ then properties \eqref{E1.1} and \eqref{E1.2} of $f$ imply that $r\circ c=\frac{\pi}{2}=R_0$. Hence $c$ parameterizes the circle $\{R_0\}\times S^1$ by arclength.
\item[(3.3)] If $0<|I(\dot{c})|<1$ let $r_-\in \left(0,\frac{\pi}{2}\right)$ and $r_+\in \left(\frac{\pi}{2},\infty\right)$ be defined by $f(r_-)=f(r_+)=|I(\dot{c})|$. Using the definition of $I$ one can find 
$s_0\in \R$ and $l_0>0$ such that $c(s_0+2nl_0)\in \{r_-\}\times S^1$ and $c(s_0+(2n+1)l_0)\in \{r_+\}\times S^1$ for all $n\in \Z$, while $(r\circ c)'(s)>0$ for $s\in 
\cup_{n\in \Z} (s_0+2nl_0,s_0+(2n+1)l_0)$ and $(r\circ c)'(s)<0$ for $s\in \cup_{n\in \Z} (s_0+(2n-1)l_0,s_0+2nl_0)$. In particular, we have $c(\R)\subseteq [r_-,r_+]\times S^1$. 
\end{itemize}

Next we exhibit the injective $g_\e$-geodesics whose existence is claimed in Proposition \ref{P1.6}(ii). Let $c_0$ denote a radial $g$-geodesic intersecting the circle $\{R_1\}\times S^1$
in the points $(R_1,0)$ and $(R_1,\pi)$, and let $c_1$ denote a radial $g$-geodesic intersecting $\{R_1\}\times S^1$ in the points $(R_1,\frac{\pi}{2})$ and $(R_1,\frac{3\pi}{2})$. 
Recall that all the metrics $g_\e$ coincide with $g$ on $\{R_1\}\times S^1$. So, being orthogonal to $\{R_1\}\times S^1$ is independent of $\e$. 

\begin{lemma}\label{L3.1}
Let $c$ be a $g_\e$-geodesic that intersects $\{R_1\}\times S^1$ orthogonally at one of the points $(R_1,0),\; (R_1,\frac{\pi}{2}),\; (R_1,\pi)$, or $(R_1,\frac{3\pi}{2})$. Then 
$c(\R)=c_0(\R)$ or $c(\R)=c_1(\R)$. In particular, up to parameterization these correspond to precisely two $g_\e$-geodesics. Moreover, these $g_\e$-geodesics are injective.
\end{lemma}

\begin{proof}
By Proposition \ref{P2.2}(iii) the radial $g$-geodesic $c_0$ is a $g_\e$-geodesics for every $\e\in (-\e_0,\e_0)$ and $c_0$ intersects $\{R_1\}\times S^1$ orthogonally in the points
$(R_1,0)$ and $(R_1,\pi)$. Hence $c(\R)=c_0(\R)$ if $c$ is a $g_\e$-geodesic intersecting $\{R_1\}\times S^1$ orthogonally at $(R_1,0)$ or at $(R_1,\pi)$. Since 
$\phi(\frac{\pi}{2})=\phi(\frac{3\pi}{2})=0$ a $g_\e$-geodesic intersecting $\{R_1\}\times S^1$ orthogonally at $(R_1,\frac{\pi}{2})$ or at $(R_1,\frac{3\pi}{2})$ satisfies $c(\R)=c_1(\R)$
by Proposition \ref{P2.2}(iii),(iv), and (v).
\end{proof}

\begin{remark}
If $c$ is a $g_\e$-geodesic satisfying $c(\R)=c_0(\R)$ then, by Proposition \ref{P2.2}(iii), there exist $a\in \{-1, 1\}$, $s_0\in \R$, such that $c(s)=c_0(as+s_0)$ for all $s\in \R$. On the 
other hand, if $\e\neq 0$ and $c$ is a $g_\e$-geodesic satisfying $c(\R)=c_1(\R)$, then the segment $[R_0,R_1]\times \{\frac{\pi}{2}\}$ is part of both $c(\R)$ and $c_1(\R)$, but its
$g_\e$-length $l_\e(\frac{\pi}{2})$ is different from its $g$-length $R_1-R_0=\frac{\pi}{4}$, see \eqref{E2.5}.
\end{remark}

Before we start the formal proof of Proposition \ref{P1.6} we give a brief outline. We assume that $c$ is an injective $g_\e$-geodesic and $|\e|>0$ is small. We want to show that $c$ 
intersects $\{R_1\}\times S^1$ orthogonally at one of the points $(R_1,0)$, $(R_1,\frac{\pi}{2})$, $(R_1,\pi)$, or $(R_1,\frac{3\pi}{2})$. Then Proposition \ref{P1.6} will follow from Lemma 
\ref{L3.1}. First we will see that the ends of $c$ lie on radial geodesics, see Lemma \ref{L3.5} and \ref{L3.6}. Then, using property \eqref{E2.2} of $\phi$, we will prove that a radial geodesic 
coming in from infinity and hitting $\{R_1\}\times S^1$ at a point $(R_1,\theta)$ with $\theta\notin \{0,\frac{\pi}{2},\pi,\frac{3\pi}{2}\}$ will be deflected by $g_\e$ so that the other end of $c$ will 
not be radial, in contradiction to the preceding statement. 

The formal proof of Proposition \ref{P1.6} will rely on the following Lemmas \ref{L3.3}-\ref{L3.7}.

\begin{lemma}\label{L3.3}
There exists $\e_1\in (0,\e_0)$ such that for every $\e\in(-\e_1,\e_1)$ and every $g_\e$-geodesic $c$ satisfying $r\circ c(0)\ge R_1$ and $\dot{c}(0)\neq \partial_r|_{c(0)}$ there exists 
$s>0$ such that $r\circ c(s)=R_0$.
\end{lemma}

\begin{proof}
Otherwise we could find sequences $\e_i\to 0$ and $g_{\e_i}$-geodesics $c_i$ such that $r\circ c_i(0)\ge R_1$, $\dot{c}_i(0)\neq \partial|_{c_i(0)}$, and $r\circ c_i(s)>R_0$ for all 
$s>0$. Since the $c_i$ are $g$-geodesics as long as they are contained in $[R_1,\infty)\times S^1$ we can use (3.1)-(3.3) to find a sequence $t_i\ge 0$ such that $r\circ c_i(t_i)=R_1$ 
and $(r\circ c_i)'(t_i)\le 0$. Using Proposition \ref{P2.2}(ii) we see that a subsequence of the sequence of $g_{\e_i}$-geodesics $s\to c_i(t_i+s)$ converges compactly to a $g$-geodesic
$c$ such that $r\circ c(0)=R_1$, $(r\circ c)'(0)\le 0$, and $r\circ c(s)\ge R_0$ for all $s>0$. But, according to (3.1)-(3.3), a $g$-geodesic with these properties does not exist.
\end{proof}

\begin{lemma}\label{L3.4}
Let $c$ be an injective $g_\e$-geodesic. If $J_c=\{s\in \R|\; r\circ c(s)\le R_0\}\neq \emptyset$, then $J_c$ is an interval of length $\pi$.
\end{lemma}

\begin{proof}
Recall that $[0,R_0]\times S^1$ with the metric $g_\e$ is isometric to a round hemisphere of radius one. Since $c$ is injective it does not parameterize the boundary great circle 
$\{R_0\}\times S^1$. Hence, if $I$ is a component of $J_c$ then $c|_I$ parameterizes half of a great circle in this hemisphere. Since any two such half great circles intersect, we see
that $J_c$ is connected and an interval of length $\pi$.
\end{proof}

Combining Lemma \ref{L3.3} and \ref{L3.4} we deduce

\begin{lemma}\label{L3.5}
Let $c$ be an injective $g_\e$ -geodesic and $|\e|<\e_1$. If $s\in\R$ and $r\circ c(s)\ge R_1$, then $\dot{c}(s)\in \{- \partial_r|_{c(s)},\partial_r|_{c(s)}\}$.
\end{lemma}

\begin{proof}
Otherwise we can use Lemma \ref{L3.3} to obtain $s_-<s<s_+$ such that $r\circ c(s_-)=r\circ c(s_+)=R_0$. So Lemma \ref{L3.4} implies $[s_-,s_+]\subset J_c$, in contradiction 
to $s\notin J_c$.
\end{proof}

Next we use the fact that $g_\e$ and $g$ coincide outside $(R_0+\e_0,R_1-\e)\times S^1$ to prove:

\begin{lemma}\label{L3.6}
There exists $0<\e_2<\e_0$ such that there is no injective $g_\e$-geodesic $c\colon [0,\infty)\to [0,R_1]\times S^1$ if $\e\in(-\e_2,\e_2)$.
\end{lemma}

\begin{proof}
Otherwise there exist sequences $\e_i\to 0$ and injective $g_{\e_i}$-geodesics $c_i\colon [0,\infty)\to [0,R_1]\times S^1$. Using Lemma \ref{L3.4} and translation of the parameter we may 
assume that $c_i([0,\infty))\subset [R_0,R_1]\times S^1$ for all $i\in \N$. Let $c$ be a limit of the $g_{\e_i}$-geodesics $\tilde{c}_i(s)=c_i(i+s)$. Then $c$ is a $g$-geodesic and $c(\R)
\subset [R_0,R_1]\times S^1$. From (3.1)-(3.3) we infer that $c$ parameterizes the circle $\{R_0\}\times S^1$, in particular $\dot{c}(s)=\pm \partial_\theta|_{c(s)}$ for all $s\in \R$. 
Since $\lim_{i\to \infty} \dot{\tilde{c}}_i(0)=\dot{c}(0)=\pm\partial_\theta|_{c(0)}$, and since $g_{\e_i}$ and $g$ coincide on $[0,R_0+\e_0]\times S^1$, we see that, for almost all $i\in \N$,
the $\tilde{c}_i$ are $g$-geodesics parameterizing a great circle of the spherical metric $g|_{[0,R_1]\times S^1}$. In particular, these $\tilde{c}_i$ are periodic, in contradiction to the assumed
injectivity of the $c_i$.
\end{proof}

\begin{lemma}\label{L3.7}
There exists $\e_3\in (0,\e_0)$ such that the following holds for all $\e\in (-\e_3,\e_3)$ and every $g_\e$-geodesic $c$. If $c(0)=(R_0,\theta)$  and $|\sphericalangle (\dot{c}(0),\partial_r|_{c(0)})|
<\e_3\max|\phi|$, then there exists $l>0$ such that $c([0,l])\subset [R_0,R_1]\times S^1$ and $c(l)\in \{R_1\}\times S^1$. If, moreover, $\dot{c}(0)\neq \dot{c}_{\e,\theta}(0)$ then 
$\dot{c}(l)\neq \partial_r|_{c(l)}$.
\end{lemma}

Note that obviously $\dot{c}(l)\neq -\partial_r|_{c(l)}$.

\begin{proof}
The first part of the claim follows from the fact that $\lim_{\e\to 0} g_\e=g$ in the $C^\infty$-topology, see Proposition \ref{P2.2}(ii). To prove the second part note that Proposition \ref{P2.2}(iv)
implies that $c_{\e,\theta}$ is the only $g_\e$-geodesic in $[R_0,R_1]\times S^1$ that starts at $(R_0,\theta)$ and ends orthogonally on $\{R_1\}\times S^1$.
\end{proof}

\begin{proof}[Proof of Proposition \ref{P1.6}]
We redefine $\e_0$ as the minimum of $\e_1,\e_2,\e_3$ from Lemma \ref{L3.3}, \ref{L3.6}, and \ref{L3.7}. We consider an injective $g_\e$-geodesic $c$ assuming $0<|\e|<\e_0$. 
We will show that $c$ intersects $\{R_1\}\times S^1$ orthogonally at one of the points $(R_1,0)$, $(R_1,\frac{\pi}{2})$, $(R_1,\pi)$, $(R_1,\frac{3\pi}{2})$. Then our claim will follow
from Lemma \ref{L3.1}. 

Combining Lemma \ref{L3.5} and \ref{L3.6} we conclude that the ends of $c$ lie on radial geodesics. In particular, we find $s_-\in \R$ and $\theta_-\in [0,2\pi)$ so that 
$c(s)=(R_1+(s_- -s),\theta_-)$ for all $s\le s_-$. We will complete the proof by showing that $\theta_-\in \{0,\frac{\pi}{2},\pi,\frac{3\pi}{2}\}$. 

Case (a): $\theta_-\in [\pi,2\pi)$. Since $\phi_\e(\theta_-)=0$ we infer from Proposition \ref{P2.2}(iii) and (iv) that, for $s\ge s_-$, $c(s)$ continues on a radial geodesic until it intersects
$\{R_0\}\times S^1$ for the second time at the parameter value $s_+=s_-+(R_1-R_0)+\pi$ and at the point $c(s_+)=(R_0,\theta_+)$, where $\theta_+=\theta_- -\pi\in [0,\pi)$. 
Note that $\dot{c}(s_+)=\partial_r|_{c(s_+)}$. We will now show that the assumption $\theta_+\notin \{0,\frac{\pi}{2}\}$ leads to a contradiction, and thereby complete the proof in case (a). 
If $\theta_+\in [0,\pi)\setminus \{0,\frac{\pi}{2}\}$ then $\phi_\e(\theta_+)\neq 0$, see \eqref{E2.3}. Hence we have $\dot{c}(s_+)=\partial_r|_{c(s_+)}\neq \dot{c}_{\e,\theta_+}(0)$. 
So Lemma \ref{L3.7} provides $l>0$ such that $r\circ c(s_+ +l)=R_1$ and $\dot{c}(s_++l)\notin \{\partial_r|_{c(s_+ +l)}, -\partial_r|_{c(s_+ +l)}\}$, in contradiction to Lemma \ref{L3.5}. 

Case (b): $\theta_- \in [0,\pi)$. We will assume that $\theta_-\notin \{0,\frac{\pi}{2}\}$ and show that this leads to a contradiction. Since $\dot{c}(s_-)=-\partial_r|_{(R_1,\theta_-)}$ 
we can use Proposition \ref{P2.2}(iv) to conclude that $c(s_- + l_\e(\theta_-))=(R_0,\theta_-)$ and 
\[
\dot{c}(s_- + l_\e(\theta_-))=-\dot{c}_{\e,\theta_-}(0)=-\cos\phi_\e(\theta_-)\partial_r|_{(R_0,\theta_-)}-\sin \phi_\e(\theta_-)\partial_\theta|_{(R_0,\theta_-)},
\]
where $0<\cos \phi_\e(\theta_-)<1$, cf. \eqref{E2.3}. We set $s_+=s_- +l_\e(\theta_-)+\pi$ and $\theta_+=\theta_-+\pi$. Since $c|_{[s_-+l_\e(\theta_-),s_-+l_\e(\theta_-)+\pi]}$ parameterizes
half of a great circle in $[0,R_0]\times S^1$ we see that $c(s_+)=(R_0,\theta_+)$ and 
\[
\dot{c}(s_+)=\cos\phi_\e(\theta_-)\partial_r|_{c(s_+)}-\sin \phi_\e(\theta_-)\partial_\theta|_{c(s_+)}\neq \partial_r|_{c(s_+)}.
\]
Since $\theta_+\in [\pi,2\pi)$ we have $\phi_\e(\theta_+)=0$, hence 
\[
\dot{c}(s_+)\neq \partial_r|_{c(s_+)}=\dot{c}_{\e,\theta_+}(0).
\]
Since $|\sphericalangle(\dot{c}(s_+),\partial_r|_{c(s_+)})|=|\phi_\e(\theta_-)|<\e_3\max |\phi|$, we can use Lemma \ref{L3.7} to find $l>0$ such that $r\circ c(s_++l)=R_1$ and 
$\dot{c}(s_++l)\notin \{\partial_r|_{c(s_++l)}, -\partial_r|_{c(s_++l)}\}$, in contradiction to Lemma \ref{L3.5}.
\end{proof}

\section{Historical remarks and open problems}

For noncompact, complete Riemannian surfaces that are not homeomorphic to the plane, the cylinder, or the M\"obius band, strong results concerning existence and quantity of 
bounded, oscillating, or proper geodesics have been given in \cite{Hadamard1898} and \cite{Woj82}. In the case of the cylinder there is the following open problem that is 
analogous to Conjecture \ref{C1.5}, but maybe a little simpler:

\begin{conjecture}
On every complete Riemannian surface $S$ homeomorphic to the cylinder $S^1\times \mathbb{R}$ there exist two disjoint embedded geodesic lines.
\end{conjecture}

Note that there is always a straight line connecting the ends of $S$, and one would expect the existence of a min-max embedded geodesic line in its complement. 

The existence of the total curvature $\int_M K\,dA$ of a complete Riemannian plane $M$ is a strong condition that helps to control the behaviour of geodesics, cf. \cite{ShShTa03}. 
From S. Cohn-Vossen's famous result \cite[Satz 6]{Co-Vo35}, one knows that $\int_M K\, dA\in [-\infty,2\pi]$ if $\int_M K\, dA$ exists as an extended real number. Note that 
$\int_M K\, dA=2\pi$ for the surfaces $(\R^2,g_\e)$ considered in the previous sections. 

\begin{thm}\label{T2}
Let $M$ be a complete Riemannian plane without simple closed geodesics.
\begin{itemize}
\item[(a)] Then there exists an injective geodesic through every point of $M$.
\item[(b)] If, additionally, $\int_M K\, dA$ exists, then there is an embedded geodesic line through every point of $M$.
\end{itemize}
\end{thm}

Note that the condition $\int_M K^+\, dA<2\pi$ implies both, the non-existence of simple closed geodesics and the existence of $\int_M K\, dA$. 

Statement (a) is proved in \cite[Theorem 2]{BanCom81}, and relies on an idea from  \cite{Co-Vo36}. Statement (b) follows from (a) combined with \cite[Theorem 3]{BanInv81}.
It is an open question if (b) is true without the additional assumption that $\int_M K\, dA$ exists. 

Combining \cite[Theorem 3.5.2 (1)]{ShShTa03}, with \cite[Theorem 3.7.4]{ShShTa03}, one obtains the following strong result on the existence of straight lines that is reminiscent 
of ``visibility results'' in the case of non-positive sectional curvature. 

\begin{thm}
Suppose $M$ is a complete Riemannian plane and $\int_M K\, dA<0$. Then there exist uncountably many straight lines in $M$. 
\end{thm}

Note that complete Riemannian planes $M$ with $\int_M K\, dA>0$ do not admit straight lines by \cite[Satz 5]{Co-Vo36}. More recently A. Carlotto and C. De Lellis \cite{Cardelell19}
employed a new min-max method to find embedded geodesic lines. They consider complete Riemannian planes $M$ of non-negative curvature that are asymptotic to a conical surface
in a strong sense. In particular, these planes have total curvature $\int_M K\, dA\in [0,2\pi)$. They find uncountably many embedded geodesic lines of min-max type. In particular these lines 
have Morse index one if the curvature is positive everywhere. Moreover, in contrast to the pure existence result Theorem \ref{T2}, one has very precise control over the asymptotic 
behaviour of these lines. 

Note that complete Riemannian planes $M$ satisfying $0<\int_M K\, dA<2\pi$ are asymptotic to a cone over a circle of length smaller than $2\pi$ with respect to pointed 
Gromov-Hausdorff convergence, see \cite[Theorem 3.7.2]{ShShTa03}. Although this type of convergence is much weaker than the one required in \cite{Cardelell19}, the following 
question seems natural: 

\begin{question}
Do the methods from \cite{Cardelell19} generalize to the case of general complete Riemannian planes $M$ with $0<\int_M K\, dA<2\pi$?
\end{question}

Finally we mention the following question which came up during the work on this paper, see also \cite{su21}.

\begin{question}
Does there exist a Riemannian metric on $S^2$ with an injective geodesic $c\colon [0,\infty)\to S^2$ that is not asymptotic to a simple closed geodesic?
\end{question}

\appendix
\section{Proof of Proposition \ref{P2.2}}

We consider {\it $t$-foliations} of $[R_0,R_1]\times S^1$, i.e. $1$-dimensional foliations of $[R_0,R_1]\times S^1$ that are transverse to all the circles $\{r\}\times S^1$, $R_0\le r\le R_1$. 
The leaves of a $t$-foliation are arcs joining $\{R_0\}\times S^1$ to $\{R_1\}\times S^1$ that are graphs of functions $[R_0,R_1]\to S^1$. The simplest example $\overline{\mathfrak{F}}$
of a $t$-foliations is generated by the vectorfield $\partial_r|_{[R_0,R_1]\times S^1}$and has leaves $[R_0,R_1]\times \{\theta\}$, $\theta\in S^1$. The diffeomorphisms $\Phi_\e$ in 
Proposition \ref{P2.2} will arise from $t$-foliations $\mathfrak{F}_\e$ by appropriately parameterizing the leaves of $\mathfrak{F}_\e$. 

We note that every $t$-foliation $\mathfrak{F}$ of $[R_0,R_1]\times S^1$ determines a map $S_\mathfrak{F}\colon [R_0,R_1]\to \text{Diff}^+(S^1)$, where $S_\mathfrak{F}(r)(\theta)=\theta'$
iff $(r,\theta')$ is the point in which the leaf of $\mathfrak{F}$ through $(R_0,\theta)$ intersects $\{r\}\times S^1$. In particular, we have $S_\mathfrak{F}(R_0)=\text{id}_{S^1}$.
The map $(r,\theta)\in [R_0,R_1]\times S^1\to S_\mathfrak{F}(r)(\theta)$ is $C^\infty$. Conversely, given a map $S\colon [R_0,R_1]\to \text{Diff}^+(S^1)$ such that $S(R_0)=\text{id}_{S^1}$
and such that $(r,\theta)\to S(r)(\theta)$ is $C^\infty$, we obtain a $t$-foliation $\mathfrak{F}$ such that $S_\mathfrak{F}=S$. 

We will use the following elementary fact about $\text{Diff}^+(S^1)$: 

\begin{fact}
The subgroup $\text{Diff}^+_0(S^1)=\{f\in \text{Diff}^+(S^1)|\; f(0)=0\}$ is contractible. 
\end{fact}

Indeed, if $f\in \text{Diff}^+_0(S^1)$ then $f$ can be lifted to a $C^\infty$-map $\tilde{f}\colon \R\to \R$ satisfying $\tilde{f}(0)=0$, $\tilde{f}'>0$, and $\tilde{f}(x+2\pi)=\tilde{f}(x)+2\pi$
for all $x\in \R$. Conversely, every such $\tilde{f}$ descends to some $f\in \text{Diff}^+_0(S^1)$. Given $f_0,f_1\in \text{Diff}^+_0(S^1)$ with lifts $\tilde{f}_0,\tilde{f}_1$ and $\lambda
\in [0,1]$, we consider the affine combination $\tilde{f}_\lambda=(1-\lambda)\tilde{f}_0+\lambda \tilde{f}_1$. Then $\tilde{f}_\lambda$ descends to $f_\lambda\in \text{Diff}^+_0(S^1)$ 
and defines a curve $\lambda\colon [0,1]\to f_\lambda \in \text{Diff}^+_0(S^1)$ from $f_0$ to $f_1$ such that $(\lambda,\theta)\to f_\lambda(\theta)$ is $C^\infty$.
We will denote $f_\lambda$ by $(1-\lambda)f_0+\lambda f_1$. The following fact \eqref{EA.1} is obvious:
\begin{equation}\label{EA.1}
\text{\parbox{.85\textwidth}{If $f_0,f_1\in \text{Diff}^+_0(S^1)$ and $f_0(\theta)=f_1(\theta)$ for some $\theta\in S^1$
then $f_\lambda(\theta)=f_0(\theta)$ for all $\lambda \in [0,1]$.}}
\end{equation}

In the following we fix an arbitrary $\delta\in \left(0,\frac{R_1-R_0}{2}\right)$, for definiteness one may take $\delta=\frac{R_1-R_0}{3}=\frac{\pi}{12}$. 

\begin{lemma}\label{LA.1}
Suppose $\mathfrak{F}_0$ and $\mathfrak{F}_1$ are $t$-foliations of $[R_0,R_1]\times S^1$ and suppose $[R_0,R_1]\times \{0\}$ is a leaf of both $\mathfrak{F}_0$ and
$\mathfrak{F}_1$. Then there exists a $t$-foliation $\mathfrak{F}$ of $[R_0,R_1]\times S^1$ that coincides with $\mathfrak{F}_0$ on $[R_0,R_0+\delta]\times S^1$ and 
with $\mathfrak{F}_1$ on $[R_1-\delta,R_1]\times S^1$. If $L$ is a leaf of both $\mathfrak{F}_0$ and $\mathfrak{F}_1$ then $L$ is a leaf of $\mathfrak{F}$.
Finally if $\mathfrak{F}_0$ and $\mathfrak{F}_1$ depend smoothly on a parameter $\e\in \R$ then so does $\mathfrak{F}$.
\end{lemma}

\begin{proof}
Since $[R_0,R_1]\times \{0\}$ is a leaf of $\mathfrak{F}_0$ and $\mathfrak{F}_1$ the maps $S_0,S_1$ corresponding to $\mathfrak{F}_0,\mathfrak{F}_1$ have range in 
$\text{Diff}^+_0(S^1)$. We choose a function $\lambda\in C^\infty([R_0,R_1],[0,1])$ such that $\lambda|_{[R_0,R_0+\delta]}=0$, $\lambda|_{[R_1-\delta,R_1]}=1$, and
define $S\colon [R_0,R_1]\to \text{Diff}^+_0(S^1)$ by 
\[
S(r)=(1-\lambda(r))S_0(r)+\lambda(r)S_1(r),
\]
and let $\mathfrak{F}$ be the $t$-foliation corresponding to $S$. Using \eqref{EA.1} we see that $\mathfrak{F}$ satisfies our claims.
\end{proof}

The following lemma is a direct consequence of the differential equation characterizing geodesics. 

\begin{lemma}\label{LA.2}
Let $g=(g_{ij})$ be a Riemannian metric on an open subset $U\subset \R^n$. Then the following conditions are equivalent:
\begin{itemize}
\item[(a)] All coordinate lines $x_1\to (x_1,\underline{x})\in U$ are (unit-speed) geodesics.
\item[(b)] $g_{11}=1$ and $\partial_1 g_{1j}=0$ for $1< j\le n$.
\end{itemize}
\end{lemma}

\begin{proof}[Proof of Proposition \ref{P2.2}] $ $\\ 
Step 1: Construction of an interpolating $t$-foliation $\mathfrak{F}_\e$. 

For $|\e|<\frac{\pi}{4\max|\phi|}$ and $\theta\in S^1$ we consider the $g$-geodesics $c^g_{\e,\theta}$ with initial vector
\begin{align}\label{EA.2}
\dot{c}^g_{\e,\theta}(0)=cos \phi_\e(\theta) \partial_r|_{(R_0,\theta)}+\sin \phi_\e(\theta) \partial_\theta|_{(R_0,\theta)}.
\end{align}
So $c^g_{\e,\theta}(0)=(R_0,\theta)$ and $(r\circ c^g_{\e,\theta})'(0)=\cos\phi_\e(\theta)>0$. If $\e=0$ then $c^g_{\e,\theta}(s)=(R_0+s,\theta)$. Using this and standard calculus
we find $\e_0\in\left(0,\frac{\pi}{4\max|\phi|}\right)$ such that the following is true for $|\e|< \e_0$: 
There exists $l^g_\e\in C^\infty(S^1,[\frac{\pi}{4},\infty))$ such that the $g$-geodesic arcs $c^g_{\e,\theta}([0,l^g_\e(\theta)])$, $\theta\in S^1$, are the leaves of a $t$-foliation 
$\mathfrak{F}^g_\e$ of $[R_0,R_1]\times S^1$. Note that $\phi_\e(\theta)=0$ implies that $L^g_\e(\theta)=R_1-R_0=\frac{\pi}{4}$ and $c^g_{\e,\theta}([0,\frac{\pi}{4}])=
[R_0,R_1]\times \{\theta\}$. In particular, $\mathfrak{F}^g_0$ is the canonical foliation $\overline{\mathfrak{F}}$. Using Lemma \ref{LA.1} we find a $t$-foliation $\mathfrak{F}_\e$
such that $\mathfrak{F}_\e$ coincides with $\mathfrak{F}^g_\e$ on $[R_0,R_0+\delta]\times S^1$ and with $\overline{\mathfrak{F}}=\mathfrak{F}^g_0$ on $[R_1-\delta,R_1]\times
S^1$, and such that $\mathfrak{F}_\e$ depends smoothly on $\e$. Since $S_{\mathfrak{F}_\e}(R_1)=S_{\overline{\mathfrak{F}}}(R_1)=\text{id}_{S^1}$, the leaf of $\mathfrak{F}_\e$ 
starting at $(R_0,\theta)$ ends at $(R_1,\theta)$. Moreover, we have 
\begin{equation}\label{EA.3}
\text{\parbox{.85\textwidth}{If $\phi_\e(\theta)=0$, then $[R_0,R_1]\times \{\theta\}$ is a leaf of $\mathfrak{F}_\e$.}}
\end{equation}

\noindent Step 2: Construction of the diffeomorphism $\Phi_\e$ from Proposition \ref{P2.2}.

We parametrize the leaves of $\mathfrak{F}_\e$ so as to obtain a diffeomorphism 
\[
\Phi_\e \colon D_\e=\{(s,\theta)|\; \theta\in S^1,\, 0\le s\le l_\e(\theta)\}\to [R_0,R_1]\times S^1
\]
with the following properties \eqref{EA.4}-\eqref{EA.7}:
\begin{equation}\label{EA.4}
\text{\parbox{.85\textwidth}{$\Phi_\e$ is smooth in the variables $(\e,s,\theta)$.}}
\end{equation}
\begin{equation}\label{EA.5}
\text{\parbox{.85\textwidth}{ $\Phi_0(s,\theta)=(R_0+s,\theta)$ for $(s,\theta)\in D_0=\left[0,\frac{\pi}{4}\right]\times S^1$.}}
\end{equation}
\begin{equation}\label{EA.6}
\text{\parbox{.85\textwidth}{$\Phi_\e(s,\theta)=c^g_{\e,\theta}(s)$ for $s\in [0,\delta]$.}}
\end{equation}
\begin{equation}\label{EA.7}
\text{\parbox{.85\textwidth}{$\Phi_\e(s,\theta)=(R_1+s-l_\e(\theta),\theta)$ for $s\in [l_\e(\theta)-\delta,l_\e(\theta)]$.}}
\end{equation}

By \eqref{EA.4} and \eqref{EA.5} the claims made in Proposition \ref{P2.2}(i) and (ii) are satisfied by $\Phi_\e$. Reducing $\e_0>0$, if necessary, we may assume that the following holds for 
$|\e|<\e_0$:
\begin{equation}\label{EA.8}
\Phi_\e([0,\delta]\times S^1)\supset\left[R_0,R_0+\frac{\delta}{2}\right]\times S^1
\end{equation}

Since $\mathfrak{F}_\e$ is a $t$-foliation we have
\begin{equation}\label{EA.9}
\text{\parbox{.85\textwidth}{$\frac{\partial(r\circ\Phi_\e)}{\partial s}(s,\theta)>0$ for all $(s,\theta)\in D_\e$.}}
\end{equation}

If $\theta\in [\pi,2\pi]$ then $\phi_\e(\theta)=0$ and $l_\e(\theta)=\frac{\pi}{4}$, cf. \eqref{E2.5}. Then $[R_0,R_1]\times \{\theta\}$ is a leaf of both $\mathfrak{F}^g_\e$ and 
$\overline{\mathfrak{F}}$, hence of $\mathfrak{F}_\e$. Consequently, we may assume 
\begin{equation}\label{EA.10}
\text{\parbox{.85\textwidth}{ If $(s,\theta)\in \left[0,\frac{\pi}{4}\right]\times [\pi,2\pi]$ then $\Phi_\e(s,\theta)=(R_0+s,\theta)$.}}
\end{equation}

Similarly, since $\phi_\e(\frac{\pi}{2})=0$, we have:
\begin{equation}\label{EA.11}
\Phi_\e\left(\left[0,l_\e\left(\frac{\pi}{2}\right)\right]\times \left\{\frac{\pi}{2}\right\}\right)=[R_0,R_1]\times \left\{\frac{\pi}{2}\right\}
\end{equation}

\noindent Step 3: Definition of the Riemannian metrics $g_\e$.

It will be technically convenient to reverse the direction of the curves $s\to \Phi_\e(s,\theta)$ and to consider $\tilde{\Phi}_\e\colon D_\e\to [R_0,R_1]\times S^1$, 
$\tilde{\Phi}_\e(s,\theta)=\Phi_\e(l_\e(\theta)-s,\theta)$, where $\Phi_\e$ is as in Step 2. Note that \eqref{EA.7} and \eqref{EA.6} transform into 
\begin{equation}\label{EA.12}
\text{\parbox{.85\textwidth}{$\tilde{\Phi}_\e(s,\theta)=(R_1-s,\theta)$ if $(s,\theta)\in [0,\delta]\times S^1$,}}
\end{equation}
and 
\begin{equation}\label{EA.13}
\text{\parbox{.85\textwidth}{ $\tilde{\Phi}_\e(l_\e(\theta)-s,\theta)=c^g_{\e,\theta}(s)$ if $(s,\theta)\in [0,\delta]\times S^1$.}}
\end{equation}

First we will show:
\begin{equation}\label{EA.14}
\text{\parbox{.85\textwidth}{ $g\left.\left(\frac{\partial\tilde{\Phi}_\e}{\partial s},\frac{\partial\tilde{\Phi}_\e}{\partial \theta}\right)\right|_{(s,\theta)}=0$ for all $(s,\theta)\in ([0,\delta]\cup [l_\e(\theta)-\delta,l_\e(\theta)])
\times S^1$.}}
\end{equation}

If $(s,\theta)\in [0,\delta]\times S^1$ then $g\left.\left(\frac{\partial\tilde{\Phi}_\e}{\partial s},\frac{\partial\tilde{\Phi}_\e}{\partial \theta}\right)\right|_{(s,\theta)}=0$ is a direct consequence of \eqref{EA.12}. 
Since $\tilde{\Phi}_\e(l_\e(\theta),\theta)=(R_0,\theta)$ by \eqref{EA.2} and \eqref{EA.13} we have
\[
l_\e'(\theta)\frac{\partial\tilde{\Phi}_\e}{\partial s}(l_\e(\theta),\theta)+\frac{\partial\tilde{\Phi}_\e}{\partial \theta}(l_\e(\theta),\theta)=\partial_\theta|_{(R_0,\theta)}.
\]
Using \eqref{EA.2}, \eqref{EA.13}, and \eqref{E2.4}, we obtain
\begin{align*}
g\left.\left(\frac{\partial\tilde{\Phi}_\e}{\partial s},\frac{\partial\tilde{\Phi}_\e}{\partial \theta}\right)\right|_{(l_\e(\theta),\theta)}&=g(-\dot{c}^g_{\e,\theta}(0),\partial_\theta|_{(R_0,\theta)}
+l_\e'(\theta) \dot{c}^g_{\e,\theta}(0))\\
&=-l_\e'(\theta)-g(\dot{c}^g_{\e,\theta}(0),\partial_\theta|_{(R_0,\theta)})=-l_\e'(\theta)-\sin\phi_\e(\theta)=0.
\end{align*}

Since the parameter lines $s\in [l_\e(\theta)-\delta,l_\e(\theta)]\to \tilde{\Phi}_\e(s,\theta)$ are $g$-geodesics, Lemma \ref{LA.2} and the preceding equation imply that 
$g\left.\left(\frac{\partial\tilde{\Phi}_\e}{\partial s},\frac{\partial\tilde{\Phi}_\e}{\partial \theta}\right)\right|_{(s,\theta)}=0$ for all $(s,\theta)\in [l_\e(\theta)-\delta,l_\e(\theta)]\times S^1$.
This completes the proof of \eqref{EA.14}. 

Now we define $g_\e$ on $[R_0,R_1]\times S^1$ by 
\begin{equation}\label{EA.15}
g_\e\left(\frac{\partial\tilde{\Phi}_\e}{\partial s},\frac{\partial\tilde{\Phi}_\e}{\partial s}\right)=1
\end{equation}
\begin{equation}\label{EA.16}
g_\e\left(\frac{\partial\tilde{\Phi}_\e}{\partial s},\frac{\partial\tilde{\Phi}_\e}{\partial \theta}\right)=0
\end{equation}
\begin{equation}\label{EA.17}
g_\e\left(\frac{\partial\tilde{\Phi}_\e}{\partial \theta},\frac{\partial\tilde{\Phi}_\e}{\partial \theta}\right)=g\left(\frac{\partial\tilde{\Phi}_\e}{\partial \theta},\frac{\partial\tilde{\Phi}_\e}{\partial \theta}\right)
\end{equation}

Finally we will show that, provided $\e_0<\frac{\delta}{2}$, $g_\e$ satisfies all the claims made in Proposition \ref{P2.2}. Since $\tilde{\Phi}_\e$ is smooth in $(\e,s,\theta)$, see \eqref{EA.4},
equations \eqref{EA.15}-\eqref{EA.17} show that $g_\e$ is smooth in $(\e,s,\theta)$. Similarly we obtain $g_0=g$ from \eqref{EA.5}. Equations \eqref{EA.8}, \eqref{EA.10}, and \eqref{EA.12}-\eqref{EA.14}, imply that, for $|\e|<\e_0$, $g_\e$ and $g$
coincide on $([R_0,R_0+\e_0]\cup[R_1-\e_0,R_1])\times S^1\cup [R_0,R_1]\times [\pi,2\pi]$. In particular,  we can extend $g_\e$ to the whole plane by setting $g_\e$ equal to $g$
outside $[R_0,R_1]\times S^1$. Then Proposition \ref{P2.2}(iii) is true. Using Lemma \eqref{LA.2}, and \eqref{EA.15} and \eqref{EA.16}, we conclude that the curves $c_{\e,\theta}\colon [0,l_\e(\theta)]
\to [R_0,R_1]\times S^1$, $c_{\e,\theta}(s)=\Phi_\e(s,\theta)$, are $g_\e$-geodesics for all $\theta\in S^1$. Moreover, \eqref{EA.2} and \eqref{EA.6} imply that $\dot{c}_{\e,\theta}(0)= 
\cos\phi_\e(\theta)\partial_r|_{(R_0,\theta)}+\sin\phi_\e(\theta)\partial_\theta|_{(R_0,\theta)}$, while \eqref{EA.7} implies $\dot{c}_{\e,\theta}(l_\e(\theta))=\partial_r|_{(R_1,\theta)}$. Finally,
$(r\circ c_{\e,\theta})'>0$ follows from \eqref{EA.9}. This proves Proposition \ref{P2.2}(iv). Proposition \ref{P2.2}(v) is equivalent to \eqref{EA.11}.
\end{proof}

\bibliographystyle{acm}        
\bibliography{Riemannian_literature}

\end{document}